\documentclass[reqno,11pt]{amsart}

\usepackage{mathdots}
\usepackage{tikz}
\usepackage{tikz-cd}
\usepackage{amssymb}
\usepackage{amsgen}
\usepackage{amsmath}
\usepackage{amsthm}
\usepackage{mathrsfs}
\usepackage{cite}
\usepackage{amsfonts}
\usepackage{enumitem}

\newcommand{\coker}{\operatorname{coker}}

\newcommand{\pd}{\operatorname{\mathrm{pd}}}

\newcommand{\inv}{^{-1}}

\newcommand{\ov}[1]{\ensuremath{\overline {#1}}}
\newcommand{\til}[1]{\ensuremath{\widetilde {#1}}}

\newcommand{\Tor}{\operatorname{\mathrm{Tor}}\nolimits}

\usepackage{xcolor}

\newtheorem*{ThmA*}{Theorem~A}
\newtheorem{Thm}{Theorem}[section]
\newtheorem{Prop}[Thm]{Proposition}

{\theoremstyle{definition}
}
{\theoremstyle{remark}
}
\newtheorem{Cor}[Thm]{Corollary}
{\theoremstyle{remark}
}
{\theoremstyle{remark}
}

{\theoremstyle{remark}
}
{\theoremstyle{remark}
}
{\theoremstyle{remark}
}
{\theoremstyle{remark}
\newtheorem*{Claim*}{Claim}}

\numberwithin{equation}{section}

\title[Exact sequence for the relation bimodule]{A Pride-Guba-Sapir exact sequence for the relation bimodule of  an associative algebra}
\dedicatory{This paper is dedicated to the memories of Victor Guba, Steve Pride and Mark Sapir}
\author{Benjamin Steinberg}
\address[B.~Steinberg]{%
    Department of Mathematics\\
    City College of New York\\
    Convent Avenue at 138th Street\\
    New York, New York 10031\\
    USA}
\email{bsteinberg@ccny.cuny.edu}

\thanks{The author was supported by a Simons Foundation Collaboration Grant, award number 849561, and the Australian Research Council Grant DP230103184.}
\date{July 16, 2024}


\begin{document}

\begin{abstract}
Given a presentation of a monoid $M$, combined work of Pride and of Guba and Sapir provides an exact sequence connecting the relation bimodule of the presentation (in the sense of Ivanov) with the first homology of the Squier complex of the presentation, which is naturally a $\mathbb ZM$-bimodule. This exact sequence was used by Kobayashi and Otto to prove the equivalence of Pride's finite homological type (FHT) property with the homological finiteness condition bi-$\mathrm{FP}_3$.  Guba and Sapir used this exact sequence to describe the abelianization of a diagram group.   We prove here a generalization of this exact sequence of bimodules for presentations of associative algebras. Our proof is more elementary than the original proof for the special case of monoids.
\end{abstract}

\maketitle
\section{Introduction}
Let $K$ be a commutative ring.  The relation bimodule arose out of the work of Bergman and Dicks on universal derivations~\cite{BergmanDickDeriv}, although the name seems to have been coined by Ivanov~\cite{IvanovRelation}.  Let $\mathcal P=\langle X\mid R\rangle$ be a presentation of a monoid $M$.  The Squier complex $D(\mathcal P)$~\cite{PrideLow,GSDiagram} is a certain $2$-complex associated to the presentation with the property that homotopy classes of paths in the $2$-complex correspond to isotopy classes of reduced monoid pictures (i.e., duals of van Kampen diagrams) for the presentation.  The fundamental groups of the Squier complex are called diagram groups~\cite{GSDiagram} and have been intensively studied in geometric group theory because they can be viewed as generalizations of Thompson's group $F$.  Pride observed that $H_1(D(\mathcal P))$ has a natural $\mathbb ZM$-bimodule structure and denotes it $\pi_2^{(b)}(\mathcal P)$.  In the case of a group presentation, this bimodule corresponds in a natural way to $\pi_2$ of an Eilenberg-Mac Lane space for the group, whence the notation.  Pride proved~\cite[Theorem~4.2]{PrideLow} that there is an exact sequence $\pi_2^{(b)}(\mathcal P)\to \mathbb ZM\otimes_{\mathbb Z} \mathbb ZR\otimes_{\mathbb Z}\mathbb ZM\to \mathcal M^{(b)}\to 0$ where $\mathcal M^{(b)}$ is the relation bimodule of the presentation,  and he asked whether the first map is injective.  Guba and Sapir~\cite[Theorem~11.3]{GSDiagram} showed that $0\to \pi_2^{(b)}(\mathcal P)\to \mathbb ZM\otimes_{\mathbb Z} \mathbb ZR\otimes_{\mathbb Z}\mathbb ZM\to \mathcal M^{(b)}\to 0$ is indeed exact, and used this to give a description of the abelianization of a diagram group.

 The proof of Guba and Sapir is somewhat indirect.  It suffices to show that the composition of the Hurewicz map on each diagram group with $H_1(D(\mathcal P))\to \mathbb ZM\otimes_{\mathbb Z} \mathbb ZR\otimes_{\mathbb Z}\mathbb ZM$  has kernel the commutator subgroup.  Guba and Sapir first assume that $R$ is a complete rewriting system and use their presentation for diagram groups of complete rewriting systems. Then they reduce the general case to this case using the Knuth-Bendix completion process and the behavior of diagram groups under Knuth-Bendix completion.  Kobayashi and Otto~\cite{KobOttoExacbi} later gave a more homological proof of the exactness of this sequence, but they still used the Knuth-Bendix completion process.

 Here I consider the more general situation of an associative $K$-algebra $A$ given by a presentation $\mathcal P=\langle X\mid R\rangle$.  I require $A$ to be projective as a $K$-module, as is the case for a monoid ring $KM$.  I define an $A$-bimodule $\pi_2(\mathcal P)$ in analogy to the construction of $\pi_2^{(b)}(\mathcal P)$.  The main result gives an exact sequence $0\to \pi_2(\mathcal P)\to A\otimes_K KR\otimes_K A\to \mathcal M\to 0$ where $\mathcal M$ is the relation bimodule of the presentation.  The proof avoids using Knuth-Bendix completion or Gr\"obner bases.  Instead, I use a long exact $\Tor$-sequence and an explicit computation of the connecting map.  Splicing this with the exact sequence $0\to\mathcal M\to A\otimes_K KX\otimes_K A\to A\otimes_K A\to A\to 0$ of Bergman and Dicks~\cite{BergmanDickDeriv}, then yields the exact sequence $0\to \pi_2(\mathcal P)\to A\otimes_K KR\otimes_K A\to A\otimes_K KX\otimes_K A\to A\otimes_K A\to A\to 0$.

 In the final section, I recover the original exact sequence of Pride-Guba-Sapir by turning the monoid presentation into an algebra presentation and checking that the two notions of $\pi_2$ give the same bimodule.  I also interpret $\pi_2^{(b)}(\mathcal P)$ as the second homology of the two-sided Cayley complex of the monoid presentation $\mathcal P$ introduced in~\cite{TopFinite1}.

\section{Relation bimodules}
In what follows $K$ is a commutative ring and $A$ is an associative $K$-algebra.  Usually, $A$ will be projective as a $K$-module.  By an $A$-bimodule, we mean a left $A\otimes_K A^{op}$-module.  We say that $A$ is \emph{bi-$\mathrm{FP}_n$} over $K$ if $A$ has a projective resolution as an $A$-bimodule that is finitely generated through degree $n$. We record here some elementary facts.

\begin{Prop}\label{p:tensor.J}
Let $A$ be a $K$-algebra and $I$ an ideal.
\begin{enumerate}
\item If $B$ is a left $A$-module, $A/I\otimes_A B\cong B/IB$.
\item If $B$ is a right $A$-module, $B\otimes_A A/I\cong B/BI$.
\item If $B$ is an $A$-bimodule, then $A/I\otimes_A B\otimes_A A/I\cong B/(IB+BI)$.
\item If $B$ is a flat right (left) $A$-module, then $B\otimes_A I\cong BI$ ($I\otimes_A B\cong IB$).
\end{enumerate}
\end{Prop}
\begin{proof}
The first two items are standard.  The third follows from these as $(B/BI)/I(B/BI)=(B/BI)/((IB+BI)/BI)\cong B/(IB+BI)$. The final follows from tensoring $B$ with the exact sequence $0\to I\to A\to A/I\to 0$ and using the first/second items.
\end{proof}

 The free monoid on a set $X$ is denoted by $X^*$, and $KX^*$ is then the free $K$-algebra generated by $X$.  Following Ivanov~\cite{IvanovRelation}, if $A$ is $X$-generated and $J=\ker (KX^*\to A)$, then $J/J^2$ is an $A$-bimodule,  termed the \emph{relation bimodule}.  Note that if $J=(R)$, then $J/J^2$ is generated by the $r+J^2$ with $r\in R$ as an $A$-bimodule.  The relation bimodule was first studied by Bergman and Dicks~\cite{BergmanDickDeriv}, but not under that name.

Let $I=\ker (KX^*\otimes_K KX^*\xrightarrow{\mu} KX^*)$ be the kernel of the multiplication map.  It is well known that $I$ is a free $KX^*$-bimodule with basis the elements $x\otimes 1-1\otimes x$ with $x\in X$. We give a quick sketch of the proof for completeness.

\begin{Prop}\label{p:derivation.free}
Let $I=\ker (KX^*\otimes_K KX^*\xrightarrow{\mu} KX^*)$.  Then $I$ is a free $KX^*$-bimodule with basis the elements $x\otimes 1-1\otimes x$ with $x\in X$.
\end{Prop}
\begin{proof}
First we show that $I$ is generated by these elements.  There is an isomorphism $K[X^*\times X^*]\to KX^*\otimes_K KX^*$ given by $(u,v)\mapsto u\otimes v$ for $u,v\in X^*$.  Composing this with $\mu$ is the $K$-linear extension of the multiplication map $X^*\times X^*\to X^*$, which  clearly has kernel with $K$-basis all differences $(u,v)-(1,uv)$ with $u,v\in X^*$ and $u\neq 1$.   It follows that $I$ has $K$-basis all elements of the form $u\otimes v-1\otimes uv$ with $u\neq 1$.  But $u\otimes v-1\otimes uv=(u\otimes 1-1\otimes u)v$, and so $I$ is generated by all elements of the form $w\otimes 1-1\otimes w$ with $1\neq w\in X^*$ as a $KX^*$-bimodule.  We show by induction on $|w|$ that $w\otimes 1-1\otimes w$ is in the sub-bimodule generated by the $x\otimes 1-1\otimes x$ with $x\in X$ with the case $|w|=1$ trivial.  If $w=xu$ with $x\in X$, then $w\otimes 1-1\otimes w=xu\otimes 1-1\otimes xu = x(u\otimes 1-1\otimes u)+(x\otimes 1-1\otimes x)u$ and the claim follows by induction.

Suppose now that $B$ is a $KX^*$-bimodule and $f\colon X\to B$ is any map.  Then \[A=\left\{\begin{bmatrix} a & b\\ 0 & a\end{bmatrix} \mid a\in KX^*, b\in B\right\}\] is a $K$-algebra, and we can define a homomorphism $\psi\colon KX^*\to A$ such that \[\psi(x) = \begin{bmatrix} x & f(x)\\ 0 &x\end{bmatrix}, \quad \psi(a) = \begin{bmatrix} a & d(a)\\ 0 &a\end{bmatrix}\] for $x\in X$ and $a\in KX^*$.  One easily verifies that $d$ is a derivation, meaning that $d(aa') = ad(a')+d(a)a'$, and that $d(x)=f(x)$. Since $(a,a')\mapsto d(a)a'$ is clearly $K$-bilinear, we can define a $K$-linear map $KX^*\otimes_K KX^*\to B$ by $\alpha(a\otimes a') = d(a)a'$.  We claim that $\alpha\colon I\to B$ is a $KX^*$-bimodule homomorphism with $\alpha(x\otimes 1-1\otimes x)=f(x)$ for $x\in X$.  This will complete the proof.  Indeed, if $\sum u_i\otimes v_i\in I$, then $\sum u_iv_i=0$.  Trivially, $\alpha(\sum u_i\otimes v_iw) = \sum d(u_i)v_iw=d(\sum u_i\otimes v_i)w$.  On the other hand, $\alpha(\sum wu_i\otimes v_i) = \sum d(wu_i)v_i = w\sum d(u_i)v_i+ d(w)\sum u_iv_i = wd(\sum u_i\otimes v_i)$.  Finally, $\alpha(x\otimes 1-1\otimes x) = d(x)-d(1)x= f(x)$, as $d(1)=0$.
\end{proof}

 The following result is due to Bergman and Dicks~\cite{BergmanDickDeriv}.  We give a proof for completeness.

\begin{Thm}\label{t:Bergman.dicks}
Let $A$ be a $K$-algebra generated by $X$ that is projective over $K$.   Then there is an exact sequence
\[0\to J/J^2\xrightarrow{d_2} A\otimes_K KX\otimes_K A\xrightarrow{d_1} A\otimes_A A\xrightarrow{\mu} A\to 0\]
where $\mu$ is the multiplication map, $d_1(1\otimes x\otimes 1) = x\otimes 1-1\otimes x$ and $d_2$ is induced by the derivation $KA^*\to  A\otimes_K KX\otimes_K A$ given by $x_1\cdots x_m\mapsto  \sum_{i=1}^m (x_1\cdots x_{i-1}+J)\otimes x_i\otimes (x_{i+1}\cdots x_m+J)$.
\end{Thm}
\begin{proof}
Since $I$ is freely generated as a $KX^*$-bimodule by the elements of the form $x\otimes 1-1\otimes x$ with $x\in X$, we have an exact sequence
\begin{equation}\label{eq:sequence}
0\to KX^*\otimes_K KX\otimes_K KX^*\to KX^*\otimes_K KX^*\to KX^*\to 0
\end{equation}
 where the first map takes $1\otimes x\otimes 1$ to $x\otimes 1-1\otimes x$ and the second map is the multiplication map.  Since $KX^*$ is free as a right $KX^*$-module, this sequence splits over $(KX^*)^{op}$, and so tensoring on the right with $A$ yields the exact sequence $0\to KX^*\otimes_K KX\otimes_K A\to  KX^*\otimes_K A\to A\to 0$.  Since $A$ is projective over $K$, it follows that $KX^*\otimes_K KX\otimes_K A$ and $KX^*\otimes_K A$ are projective left $KX^*$-modules.  Thus, tensoring on the left with $A=KX^*/J$ and recalling that $\Tor_1^{KX^*}(KX^*/J,KX^*/J)\cong J/J^2$ (cf.~\cite[Exercise~19, p.~126]{CartanEilenberg}), yields the exact sequence \[0\to J/J^2\to A\otimes_K KX\otimes_K A\to A\otimes_K A\to A\to 0.\]  It is a straightforward exercise to check that the connecting map is given by $d_2$ above (or see~\cite{BergmanDickDeriv}).
\end{proof}

In particular, if $X$ is finite, then $A$ is bi-$\mathrm{FP}_2$ over $K$ if and only if the relation bimodule is finitely generated.
Next we shall need a result I learned of from Webb~\cite{WebbResolution}.

\begin{Prop}\label{p.proj.dim}
Let $M$ be a left/right $KX^*$-module that is projective over $K$.  Then $\pd M\leq 1$.
\end{Prop}
\begin{proof}
We just handle the case $M$ is a left $KX^*$-module.   Tensoring the split exact sequence of free right $KX^*$-modules  \eqref{eq:sequence} on the right with $M$ yields an exact sequence $0\to KX^*\otimes_K KX\otimes_K M\to KX^*\otimes_K M\to M\to 0$ of left $KX^*$-modules.  Since $M$ is a projective $K$-module,  $KX^*\otimes_K KX\otimes_K M$ and $KX^*\otimes_K M$ are projective $KX^*$-modules.  Therefore, $\pd M\leq 1$.
\end{proof}

As a corollary, we can deduce that $J$ is projective as both a left and right $KX^*$-module.

\begin{Cor}\label{c:J.proj}
Let $J$ be an ideal of $KX^*$ and suppose that $KX^*/J$ is projective over $K$.  Then $J$ is projective as both a left and right $KX^*$-module.
\end{Cor}
\begin{proof}
We handle the case of left $KX^*$-modules.  We have an exact sequence $0\to J\to KX^*\to KX^*/J\to 0$ with $KX^*$ free.  Also $\pd KX^*/J\leq 1$ by Proposition~\ref{p.proj.dim}.  It follows that $\pd J=0$, that is, $J$ is projective as a left $KX^*$-module.
\end{proof}

The next proposition will be used to show that $J^m$ is projective as both a left and right module.

\begin{Prop}\label{p:bimodule.proj}
Let $B,C$ be bimodules over a $K$-algebra $A$ that are projective as both left and right $A$-modules.  Then $B\otimes_A C$ is projective as both a left and right $A$-module.
\end{Prop}
\begin{proof}
Suppose that $C\oplus C'=F$ with $F$ a free left $A$-module.  Then $(B\otimes_A C)\oplus (B\otimes_A C')\cong  B\otimes_A F\cong \bigoplus B$.  Thus $B\otimes_A C$ is projective as a left module (since $B$ is).  The other claim is dual.
\end{proof}

\begin{Cor}\label{c:prod.ideal}
Let $I$ be an ideal of $K$-algebra $A$ that is projective as both a left and right $A$-module.  If $B$ is an $A$-bimodule that is projective as both a left and right $A$-module, then $IB,BI$ are projective as both left and right $A$-modules.  In particular, $I^m$ is projective as both a left and right $A$-module for all $m\geq 1$.
\end{Cor}
\begin{proof}
This follows from Proposition~\ref{p:tensor.J}(4) and Proposition~\ref{p:bimodule.proj}.
\end{proof}

\section{The homotopy bimodule}
A \emph{presentation} of a $K$-algebra  is a gadget $\mathcal P=\langle X\mid R\rangle$ where $X$ is a set and $R\subseteq KX^*$.  The presentation is finite if $X$ and $R$ are finite.  A $K$-algebra $A$ is \emph{presented} by $\mathcal P$ if $A\cong KX^*/(R)$.  As usual, we put $J=(R)$.  We shall assume throughout this section that $A$ is projective over $K$, and hence $J$ is projective as both a left and right $KX^*$-module by Corollary~\ref{c:J.proj}.    The relation bimodule $J/J^2$ is generated by the cosets $r+J$ with $r\in R$, whence $A$ is bi-$\mathrm{FP}_2$ if the presentation is finite.  It follows there is an exact sequence $A\otimes_K KR\otimes_K A\to J/J^2\to 0$, and we aim to compute the kernel.  This was done in the case that $A$ is an integral monoid ring and $\mathcal P$ comes from a monoid presentation in the work of Pride~\cite{PrideLow} and of Guba-Sapir~\cite{GSDiagram}, where it was shown that the kernel is the first homology of the Squier complex of the presentation.
Here we construct the analogue of the first homology of the Squier complex for any presentation of a $K$-algebra $A$ that is projective over $K$.  We then prove the analogue of the Pride-Guba-Sapir exact sequence.  In the next section, we shall derive anew the Pride-Guba-Sapir sequence in the monoid case from our more general result.

Since $J=(R)$, we have an exact sequence of $KX^*$-bimodules
\begin{equation}\label{eq:exact.H}
0\to H\to KX^*\otimes_K KR\otimes_K KX^*\xrightarrow{\alpha} J\to 0.
\end{equation}
  To keep $R$ and its image in $KX^*$ formally distinct, we write $\ov r$ for the basis element of $KR$ corresponding to $r\in R\subseteq KX^*$.  Then $\alpha(u\otimes \ov r\otimes v)=urv$.   Note that since $J$ is projective as both a left and right $KX^*$-module by Corollary~\ref{c:J.proj}, the sequence \eqref{eq:exact.H} splits when viewed as either a sequence of left or of right $KX^*$-modules.  Moreover, since $KX^*\otimes_K KR\otimes_K KX^*$ is free as both a left and right $KX^*$-module, we deduce that $H$ is projective as both a left and right $KX^*$-module.

If $r_i,r_j\in R$ and $u\in X^*$, then we put $[r_i,u,r_j] = 1\otimes \ov r_i\otimes ur_j-r_iu\otimes \ov {r}_j\otimes 1\in KX^*\otimes_K KR\otimes_K KX^*$.  Note that $\alpha([r_i,u,r_j]) = r_iur_j-r_iur_j=0$, and so $[r_i,u,r_j]\in H$.  Let $D$ be the sub-$KX^*$-bimodule of $H$ generated by $\{[r_i,u,r_j]\mid r_i,r_j\in R, u\in X^*\}$.  Note that $D$ is spanned as a $K$-module by the elements $v[r_i,u,r_j]w$ with $v,w\in X^*$.

\begin{Prop}\label{p:is.A.mod}
There is a containment $JH+HJ\subseteq D$, and hence $H/D$ is an $A$-bimodule.
\end{Prop}
\begin{proof}
We claim that if $r_i\in R$ and $h\in KX^*\otimes_K KR\otimes_K KX^*$, then $r_ih +D = 1\otimes \ov r_i\otimes \alpha(h)+D$ and $hr_i+D = \alpha(h)\otimes \ov r_i\otimes 1+D$.  Since $H=\ker \alpha$, it will then follow that $JH+HJ\subseteq D$.  We just prove the first equality, for which it suffices to handle the case $h=u\otimes \ov r_j\otimes v$ with $u,v\in X^*$ and $r_j\in R$.  Then $r_ih +[r_i,u,r_j]v = r_iu\otimes \ov r_j\otimes v+1\otimes \ov r_i\otimes ur_jv-r_iu\otimes \ov r_j\otimes v = 1\otimes \ov r_i\otimes \alpha(h)$, as was required.
\end{proof}

The $A$-bimodule $H/D$ will be denoted $\pi_2(\mathcal P)$ in analogy to the corresponding notation used by Pride in the case of integral monoid rings and monoid presentations.

The crucial idea is now the following.  The exact sequence \eqref{eq:exact.H} splits over $(KX^*)^{op}$, and so tensoring on the right with $A=KX^*/J$ yields an exact sequence $0\to H/HJ\to KX^*\otimes_K KR\otimes_K A\to J/J^2\to 0$.  Now tensoring on the left with $A$ and observing that $KX^*\otimes_K KR\otimes_K A$ is a projective $KX^*$-module, because $A$ is projective over $K$, yields an exact sequence $0\to \Tor_1^{KX^*}(KX^*/J,J/J^2)\to H/(JH+HJ)\to A\otimes_K KR\otimes_K A\to J/J^2\to 0$.  Note that $0\to J^2\to J\to J/J^2\to 0$ is a projective resolution of $J/J^2$ as a left $KX^*$-module by Corollary~\ref{c:prod.ideal}, and so $\Tor_1^{KX^*}(KX^*/J,J/J^2) = \ker(J^2/J^3\to J/J^2)=J^2/J^3$.  We will compute the connecting map $J^2/J^3\to H/(JH+HJ)$ and see that its image is $D/(JH+HJ)$.  This will then yield the exact sequence $0\to \pi_2(\mathcal P)\to A\otimes_K KR\otimes_K A\to J/J^2\to 0$, that we desire.

\begin{Thm}\label{t:Guba.Sapir.Pride}
Let $A$ be a $K$-algebra that is projective over $K$ given by a presentation $\mathcal P=\langle X\mid R\rangle$.  Let $J=(R)$.  Then there is an exact sequence of $A$-bimodules $0\to \pi_2(\mathcal P)\to A\otimes_K KR\otimes_K A\to J/J^2\to 0$.
\end{Thm}
\begin{proof}
Note that  $F_R=KX^*\otimes_K KR\otimes_K KX^*$ is free as a left  $KX^*$-module.
We have already observed that $H$ is projective as both a left and right $KX^*$-module.  Therefore, $HJ$ is also projective as both a left and right $KX^*$-module and $HJ\cong H\otimes_{KX^*} J$ by Corollary~\ref{c:prod.ideal} and Proposition~\ref{p:tensor.J}(4).  Thus $0\to HJ\to H\to H/HJ\to 0$ is a projective resolution as both left and right $KX^*$-modules.  Also, as already mentioned, $0\to J^2\to J\to J/J^2\to 0$ is a projective resolution.  Clearly,  $0\to F_RJ\to F_R\to F_R/F_RJ\to 0$ is exact.  Moreover,    $F_R/F_RJ\cong KX^*\otimes_K KR \otimes_K A$ is projective as a left $KX^*$-module, since $A$ is a projective $K$-module, and so this sequence splits over $KX^*$.  Hence $F_RJ$ is also projective as a left $KX^*$-module (which also follows from Corollary~\ref{c:prod.ideal}).    Note that since $J$ is projective as a left $KX^*$-module (hence flat), tensoring \eqref{eq:exact.H} on the right with $J$ and applying Proposition~\ref{p:tensor.J}(4), we have that $0\to HJ\to F_RJ\to J^2\to 0$ is exact.  Finally, as we already observed, since \eqref{eq:exact.H} splits over $(KX^*)^{op}$, tensoring on the right with $A$ yields the exact sequence $0\to H/HJ\to F_R/F_RJ\to J/J^2\to 0$.  Thus we have the commutative diagram of left $KX^*$-modules displayed in Figure~\ref{f:fig1} with all rows exact, each column a projective resolution and with the middle column split.
\begin{figure}[tbhp]
\[\begin{tikzcd}
       & 0\ar{d}        & 0\ar{d}          & 0\ar{d}          &\\
0\ar{r}& HJ\ar{r}\ar{d} &F_RJ\ar{r}\ar{d}    & J^2\ar{r}\ar{d}  & 0\\
0\ar{r}& H\ar{r}\ar{d}  &  F_R\ar{r}\ar{d}   & J\ar{r}\ar{d}    & 0 \\
0\ar{r}&H/HJ\ar{r}\ar{d}& F_R/F_RJ\ar{r}\ar{d} & J/J^2\ar{r}\ar{d}& 0\\
       & 0              & 0                &  0.               &
\end{tikzcd}\]
\caption{\label{f:fig1}}
\end{figure}
Since each module in the first two rows is projective, these rows split.  Hence tensoring on the left with $A=KX^*/J$ yields the commutative diagram with exact rows and columns in Figure~\ref{f:fig2}.
\begin{figure}
\[\begin{tikzcd}
       &         & 0\ar{d}          &           &\\
0\ar{r}& HJ/JHJ\ar{r}\ar{d}{a} &F_RJ/JF_RJ\ar{r}\ar{d}{b}    & J^2/J^3\ar{r}\ar{d}{c}  & 0\\
0\ar{r}& H/JH\ar{r}\ar{d}  &  F_R/JF_R\ar{r}\ar{d}   & J/J^2\ar{r}\ar{d}    & 0 \\
       &H/(JH+HJ)\ar{r}\ar{d}& F_R/(JF_R+F_RJ)\ar{r}\ar{d} & J/J^2\ar{r}\ar{d}& 0\\
       & 0              & 0                &  0.               &
\end{tikzcd}\]
\caption{\label{f:fig2}}
\end{figure}
Applying the snake lemma gives us an exact sequence $0\to \ker a\to 0\to \ker c\xrightarrow{\partial} \coker a\to \coker b\to \coker c\to 0$ where $\partial$ is the connecting map and the others are induced naturally.  Notice that $c$ is the zero map, and so $\ker c=J^2/J^3$, whereas $\coker a=H/(JH+HJ)$, $\coker b=F_R/(JF_R+F_RJ)$ and $\coker c=J/J^2$.  Note that $J^2/J^3$ is generated as a left $A$-module by elements of the form $r_iur_jv+J^3$ with $u,v\in X^*$ and $r_i,r_j\in R$.  We can take as a preimage in $F_RJ/JF_RJ$ the element $1\otimes \ov r_i\otimes ur_jv+JF_RJ$ (since $ur_jv\in J$), which maps to $1\otimes \ov r_i\otimes ur_jv+JF_R = 1\otimes \ov r_i\otimes ur_jv- r_iu\otimes \ov r_j\otimes v+JF_R = [r_i,u,r_j]v+JF_R$ (since $r_iu\in J)$.  Thus $\partial (r_iur_jv+J^3) = [r_i,u,r_j]v+JH+HJ$.  It follows that the image of $\partial$ is $(D+JH+HJ)/(JH+HJ)=D/(JH+HJ)$ (by Proposition~\ref{p:is.A.mod}), and hence $\coker \partial \cong H/D=\pi_2(\mathcal P)$. Since $F_R/(JF_R+F_RJ)\cong A\otimes_K KR\otimes_K A$,  this completes the proof.
\end{proof}

Splicing together the exact sequences in Theorems~\ref{t:Bergman.dicks} and~\ref{t:Guba.Sapir.Pride} yields the following result.

\begin{Cor}\label{c:partial.res}
Let $A$ be a $K$-algebra that is projective over $K$ with presentation $\mathcal P=\langle X\mid R\rangle$.  Then there is an exact sequence of $A$-bimodules
\[0\to \pi_2(\mathcal P)\to A\otimes_K KR\otimes_K A\to A\otimes_K KX\otimes_K A\to A\otimes_K A\to A\to 0.\]  In particular, if $\mathcal P$ is a finite presentation, then $A$ is of type bi-$\mathrm{FP}_3$ over $K$ if and only if $\pi_2(\mathcal P)$ is finitely generated.
\end{Cor}

\section{Monoid presentations}
We now consider the case of a monoid presentation and apply Theorem~\ref{t:Guba.Sapir.Pride} to recover a result of Pride~\cite[Theorem~4.2]{PrideLow} and Guba-Sapir~\cite[Theorem~11.3]{GSDiagram} (see also~\cite{KobOttoExacbi}).  The proof of Theorem~\ref{t:Guba.Sapir.Pride} is more elementary than the proofs found in~\cite{GSDiagram} or~\cite{KobOttoExacbi} as we avoid Knuth-Bendix completions.

Let $\mathcal P = \langle X\mid R\rangle$ be a monoid presentation of a monoid $M$.  We view $R\subseteq X^*\times X^*$.  The \emph{derivation graph} $\Gamma(\mathcal P)$ has vertex set $X^*$ and directed edge set $X^*\times R\times X^*$ where $(u,(\ell,r),v)\colon u\ell v\to urv$.  It is evident that two vertices are in the same connected component of $\Gamma(\mathcal P)$ if and only if they represent the same element of $M$.  The \emph{Squier complex} $D(\mathcal P)$ has $1$-skeleton $\Gamma(\mathcal P)$ and has set of $2$-cells $X^*\times R\times X^*\times R\times X^*$.  The cell corresponding to $(u,(\ell,r),v,(\ell',r'),w)$ has boundary path \[(u,(\ell,r),v\ell' w)(urv,(\ell',r'),w)[(u\ell v,(\ell',r'),w)(u,(\ell,r),vr'w)]\inv.\]
 Roughly speaking, paths correspond to monoid pictures of a derivation from one word to another, and two paths are homotopic if and only if they can be obtained from each other by isotopy and removing or inserting dipoles. See~\cite{GSDiagram} for details, as we will not need monoid pictures.

Notice that $X^*$ has commuting left and right actions on $D(\mathcal P)$ by regular cellular maps induced by multiplication.

The monoid algebra $\mathbb ZM$ has presentation as a $\mathbb Z$-algebra $\mathcal P'=\langle X\mid \til R\rangle$ where $\til R= \{r-\ell\mid (\ell,r)\in R\}$.  Moreover, $\mathbb ZM$ is free over $\mathbb Z$ so the above theory applies. Put $J=(\til R)$ as usual. Then $J/J^2$ is the relation bimodule considered by Ivanov~\cite{IvanovRelation}. Notice that if we look at the cellular chain complex of $\Gamma(\mathcal P)$, we get the exact sequence $0\to H_1(\Gamma(\mathcal P))\to \mathbb Z[X^*\times R\times X^*]\xrightarrow{d_1} \mathbb ZX^*$, and the image of $d_1$ is $J$ as $d_1((1,(\ell,r),1)) = r-\ell$.  Identifying $\mathbb Z[X^*\times R\times X^*]$ with $\mathbb ZX^*\otimes_{\mathbb Z} \mathbb Z\til R\otimes_{\mathbb Z} \mathbb ZX^*$ via $(u,(\ell,r),v)\mapsto u\otimes \ov{r-\ell}\otimes v$ then identifies $H_1(\Gamma(\mathcal P))$ with $H$ from \eqref{eq:exact.H}.

On the other hand, if $d_2\colon \mathbb Z[X^*\times R\times X^*\times R\times X^*]\to \mathbb Z[X^*\times R\times X^*]$ is the boundary map in the cellular chain complex of $D(\mathcal P)$, then $d_2(u,(\ell,r),v,(\ell',r'),w) = (u,(\ell,r),v\ell'w)-(u,(\ell,r),vr'w)+(urv,(\ell',r'),w)-(u\ell v,(\ell',r'),w)$.  Under our identification of $\mathbb Z[X^*\times R\times X^*]$ with $\mathbb ZX^*\otimes_{\mathbb Z} \mathbb ZR\otimes_{\mathbb Z} \mathbb ZX^*$ this maps to $-u\otimes \ov{r-\ell}\otimes v(r'-\ell') w+u(r-\ell)v\otimes \ov{r'-\ell'}\otimes w= -u[(r-\ell,v,r'-\ell')]w$.  It follows that $H_1(D(\mathcal P))\cong H/D=\pi_2(\mathcal P')$ as $\mathbb ZX^*$-bimodules, and hence $H_1(D(\mathcal P))$ is a $\mathbb ZM$-bimodule (as was already observed by Pride).  Pride puts $\pi_2^{(b)}(\mathcal P)=H_1(D(\mathcal P))$.  Putting together these observations with Theorem~\ref{t:Guba.Sapir.Pride} recovers the exact sequence of Pride~\cite{PrideLow} and Guba-Sapir~\cite{GSDiagram}.

\begin{Thm}[Pride/Guba-Sapir]\label{t:PGS.OG}
Let $\mathcal P=\langle X\mid R\rangle$ be a monoid presentation.  Then there is an exact sequence of $\mathbb ZM$-bimodules \[0\to \pi_2^{(b)}(\mathcal P)\to \mathbb ZM\otimes_{\mathbb Z} \mathbb ZR\otimes_{\mathbb Z} \mathbb ZM\to J/J^2\to 0\] where $J/J^2$ is the relation bimodule of this presentation.
\end{Thm}

This theorem implies that the abelianization of a diagram group is free abelian~\cite[Theorem~11.2]{GSDiagram}.

There is an augmentation homomorphism $\varepsilon\colon \mathbb ZM\to \mathbb Z$ with $\varepsilon(m)=1$ for all $m\in M$.  The left (respectively, right) relation module, studied by Ivanov~\cite{IvanovRelation}, is $J/J^2\otimes_{\mathbb ZM} \mathbb Z$ (respectively, $\mathbb Z\otimes_{\mathbb ZM}J/J^2$).  Pride defined $\pi_2^{(\ell)}(\mathcal P) = \pi_2^{(b)}(\mathcal P)\otimes_{\mathbb ZM}\mathbb Z$ and $\pi_2^{(r)}(\mathcal P)=\mathbb Z\otimes_{\mathbb ZM} \pi_2^{(b)}(\mathcal P)$.  The following exact sequences are known~\cite{PrideLow,KobOttoExacbi}, but our proof is different.

\begin{Cor}\label{c:exact}
Let $\mathcal P=\langle X\mid R\rangle$ be a presentation of a monoid $M$.  Let $J=\ker(\mathbb ZX^*\to \mathbb ZM)$. Then there are exact sequences
\begin{gather*}
0\to J/J^2\otimes_{\mathbb ZM}\mathbb Z\to \mathbb ZM\otimes_{\mathbb Z}\mathbb ZX\to \mathbb ZM\to \mathbb Z\to 0\\
0\to \pi_2^{(\ell)}(\mathcal P)\to \mathbb ZM\otimes_{\mathbb Z}\mathbb ZR\to J/J^2\otimes_{\mathbb ZM}\mathbb Z\to 0
\end{gather*}
and dually for the right relation module and $\pi_2^{(r)}(\mathcal P)$.
\end{Cor}
\begin{proof}
By Theorem~\ref{t:Bergman.dicks}, we have exact sequences of $\mathbb ZM$-bimodules $0\to \Omega\to \mathbb ZM\otimes_{\mathbb  Z}\mathbb ZM\to \mathbb ZM\to 0$ and $0\to J/J^2\to \mathbb ZM\otimes_{\mathbb Z}\mathbb ZX\otimes_{\mathbb Z}\mathbb ZM\to \Omega\to 0$.  Since $\mathbb ZM$ and $\mathbb ZM\otimes_{\mathbb  Z}\mathbb ZM$ are free right $\mathbb ZM$-modules, the first sequence splits over $\mathbb ZM^{op}$ and $\Omega$ is a projective right $\mathbb ZM$-module.  Therefore the second sequence splits over $\mathbb ZM^{op}$ and $J/J^2$ is a projective right $\mathbb ZM$-module (since $\mathbb ZM\otimes_{\mathbb Z}\mathbb ZX\otimes_{\mathbb ZM}\mathbb ZM$ is free as a right $\mathbb ZM$-module).  Thus tensoring on the right with $\mathbb Z$ yields exact sequences $0\to \Omega\otimes_{\mathbb ZM}\mathbb Z\to \mathbb ZM\to \mathbb Z\to 0$ and $0\to J/J^2\otimes_{\mathbb ZM}\mathbb Z\to \mathbb ZM\otimes_{\mathbb Z}\mathbb ZX\to \Omega\otimes_{\mathbb ZM}\mathbb Z\to 0$, and these splice together to give the first exact sequence.

The exact sequence $0\to \pi_2^{(b)}(\mathcal P)\to \mathbb ZM\otimes_{\mathbb Z}\mathbb ZR\otimes_{\mathbb Z}\mathbb ZM\to J/J^2\to 0$ splits over $\mathbb ZM^{op}$ because we observed that $J/J^2$ is a projective right $\mathbb ZM$-module.  Therefore, tensoring on the right with $\mathbb Z$ yields the second exact sequence.
\end{proof}

If we splice the two exact sequences in Corollary~\ref{c:exact} together, we get the exact sequence $0\to \pi_2^{(\ell)}(\mathcal P)\to \mathbb Z[M\times R]\to \mathbb Z[M\times X]\to \mathbb ZM\to \mathbb Z\to 0$.  Recall that the (left) Cayley complex $C^{(\ell)}(\mathcal P)$ of $M$ with respect to the presentation $\mathcal P$ is the free $M$-CW complex with vertex set $M$, edge set $M\times X$ and set of $2$-cells $M\times R$~\cite{TopFinite1}.  Here $(m,x)\colon m\to mx$ and the boundary path  of $(m,(\ell,r))$ is the path from $m$ to $m\ell=mr$ labelled by $r$ followed by the inverse of the path from $m$ to $m\ell=mr$ labelled by $\ell$.  This complex is simply connected and the above exact sequence identifies $\pi_2^{(\ell)}(\mathcal P)$ as a left $\mathbb ZM$-module with $H_2(C^{(\ell)}(\mathcal P))\cong \pi_2(C^{(\ell)}(\mathcal P))$.

We observe that Theorem~\ref{t:PGS.OG} provides a similar result for the two-sided Cayley complex $C^{(b)}(\mathcal P)$~\cite{TopFinite1}.  The vertex set of $C^{(b)}(\mathcal P)$ is $M\times M$, the edge set is $M\times X\times M$ and the set of $2$-cells is $M\times R\times M$.  The edge $(m,x,m')$ goes from $(m,xm')\to (mx,m')$.  If $(\ell,r)\in R$ and $(m,m')\in M\times M$, then $m\ell=mr$, $\ell m'=rm'$, and it is not difficult to see that there are paths labelled by $\ell$ and by $r$ from $(m,\ell m')\to (m\ell,m')$.  The boundary path of the cell $(m,(\ell,r),m')$ is the aforementioned path labelled by $r$ followed by the inverse of the aforementioned path labelled by $\ell$.  There are natural commuting left and right cellular actions of $M$ on $C^{(b)}$ induced by multiplication.  The set of path components of $C^{(b)}(\mathcal P)$ is in bijection with $M$ via $(m,m')\mapsto mm'$ on vertices and each path component is simply connected.  Moreover, one has $H_0(C^{(b)}(\mathcal P))\cong \mathbb ZM$ as a $\mathbb ZM$-bimodule.

The exact sequence in Theorem~\ref{t:PGS.OG} can be spliced with that in Theorem~\ref{t:Bergman.dicks} to yield an exact sequence $0\to \pi_2^{(b)}(\mathcal P)\to \mathbb ZM\otimes_{\mathbb Z}\mathbb ZR\otimes_{\mathbb Z}\mathbb ZM\to \mathbb ZM\otimes_{\mathbb Z}\mathbb ZX\otimes_{\mathbb Z}ZM\to \mathbb ZM\otimes_{\mathbb Z}\mathbb ZM\to \mathbb ZM\to 0$.  It is not difficult to check using Theorem~\ref{t:Bergman.dicks} that this can be identified with the exact sequence
$0\to \pi_2^{(b)}(\mathcal P)\to \mathbb Z[M\times R\times M]\to \mathbb Z[M\times X\times M]\to \mathbb Z[M\times M]\to H_0(C^{(b)}(\mathcal P))\to 0$, where the maps come from the cellular chain complex of $C^{(b)}(\mathcal P)$.  Therefore,  we can identify $\pi_2^{(b)}(\mathcal P)$ with $H_2(C^{(b)}(\mathcal P))$ as a $\mathbb ZM$-bimodule.  Since each path component of $C^{(b)}(\mathcal P)$ is simply connected, as an abelian group this is the direct sum of the $\pi_2$ of the path components of $C^{(b)}(\mathcal P)$, justifying the notation. It follows from the above exact sequence that if the presentation is finite, then $M$ is bi-$\mathrm{FP}_3$ if and only if $\pi_2^{(b)}(\mathcal P)$ is finitely generated, as was observed by Kobayashi and Otto~\cite{KobayashiOtto}.

\def\malce{\mathbin{\hbox{$\bigcirc$\rlap{\kern-7.75pt\raise0,50pt\hbox{${\tt
  m}$}}}}}\def\cprime{$'$} \def\cprime{$'$} \def\cprime{$'$} \def\cprime{$'$}
  \def\cprime{$'$} \def\cprime{$'$} \def\cprime{$'$} \def\cprime{$'$}
  \def\cprime{$'$} \def\cprime{$'$}

\end{document}